\documentclass{amsart}
\usepackage{amsfonts,amssymb,amscd,amsmath,enumerate,verbatim,calc}
\usepackage[all]{xy}

\newcommand{\CM}{Cohen-Macaulay}

\newcommand{\wrt}{with respect to}

\newcommand{\n}{\mathfrak{n} }
\newcommand{\m}{\mathfrak{m} }
\newcommand{\M}{\mathfrak{M} }

\newcommand{\R}{\mathcal{R} }

\newcommand{\Sc}{\mathcal{S} }
\newcommand{\rt}{\rightarrow}

\newcommand{\grade}{\operatorname{grade}}
\newcommand{\depth}{\operatorname{depth}}

\newcommand{\Tor}{\operatorname{Tor}}

\newcommand{\codim}{\operatorname{codim}}

\newcommand{\Syz}{\operatorname{Syz}}

\theoremstyle{plain}

\newtheorem{theorem}{Theorem}[section]

\newtheorem{proposition}[theorem]{Proposition}

\theoremstyle{definition}

\newtheorem{s}[theorem]{}

\newtheorem{example}[theorem]{Example}

\theoremstyle{remark}

\begin{document}

\title[Monotonicity of Hilbert functions]{ On the Monotonicity of Hilbert functions}
\author{Tony~J.~Puthenpurakal}
\date{\today}
\address{Department of Mathematics, IIT Bombay, Powai, Mumbai 400 076}

\email{tputhen@math.iitb.ac.in}


 \begin{abstract}
 In this paper we show that a large class of one-dimensional  \CM \ local rings $(A,\m)$ has  the property that if $M$ is a maximal \CM \ $A$-module then the Hilbert function of  $M$ (\wrt \ $\m$) is non-decreasing. Examples include
 \begin{enumerate}
 \item
 Complete intersections $A = Q/(f,g)$ where $(Q,\n)$ is regular local of dimension three and $f \in \n^2 \setminus \n^3$.
 \item
 One dimensional \CM \ quotients of a two dimensional \CM \  local ring with pseudo-rational singularity.
 \end{enumerate}
\end{abstract}
 \maketitle
\section{introduction}
Let $(A,\m)$ be a $d$-dimensional Noetherian local ring  with residue field $k$ and let $M$ be a finitely generated $A$-module. Let $\mu(M)$ denote minimal number of generators of $M$ and let $\ell(M)$ denote its length. Let $\codim(A) = \mu(\m) - d$ denote the codimension of $A$.

Let $G(A) = \bigoplus_{n\geq 0}\m^n/\m^{n+1}$ be the associated graded ring of $A$ (\wrt \ $\m$) and let $G(M) = \bigoplus_{\n\geq 0}\m^nM/\m^{n+1}M$ be the associated graded module of $M$ considered as a $G(A)$-module. The ring $G(A)$ has a  unique graded maximal ideal
$\M_G =  \bigoplus_{n\geq 1}\m^n/\m^{n+1} $.
 Set $ \depth G(M) =  \grade(\M_G,G(M))$. Let $e(M)$ denote the multiplicity of $M$ (\wrt \ $\m$).

The Hilbert function of $M$ (\wrt \ $\m$) is the function
\[
H(M,n) = \ell \left( \frac{\m^nM}{\m^{n+1}M} \right) \quad \text{for all} \ n\geq 0.
\]
A  natural question is whether $H(M,n)$ is non-decreasing (when $\dim M > 0$).
It is clear that if $\depth G(M) > 0$ then the Hilbert function of $M$ is \emph{non}-decreasing, see Proposition 3.2 of \cite{Pu2}.
If $A$ is regular local then all maximal \CM \ (= MCM ) modules are free. Thus every MCM module of positive dimension over a regular local ring has a non-decreasing Hilbert function. The next case is that
of a hypersurface ring i.e., the completion $\widehat{A} = Q/(f)$ where $(Q,\n)$ is regular local and $f \in \n^2$.
In Theorem 1,\cite{Pu2} we prove that if $A$ is a hypersurface ring of positive dimension and if $M$ is a MCM $A$-module then the Hilbert function of $M$ is non-decreasing. See example 3.3, \cite{Pu2}  for an example of a MCM module $M$ over the hypersurface ring
$k[[x,y]]/(y^3)$ with $\depth G(M) = 0$. 

Let $(A,\m)$ be a strict complete intersection of positive dimension and let $M$ be a maximal \CM \ $A$-module with  bounded betti-numbers.  In Theorem 1, \cite{Pu3} we prove that the Hilbert function of $M$ is non-decreasing. We also prove an analogous statement for complete intersections of codimension two, see Theorem 2, \cite{Pu3}.

In the ring case Elias \cite[2.3]{Elias},  proved that the Hilbert function of a  one dimensional Cohen-Macaulay ring is non-decreasing if embedding dimension is three. The first example of a  one dimensional \CM \ ring $A$ with not monotone increasing Hilbert function was given by Herzog and Waldi; \cite[3d]{HW}. Later Orecchia, \cite[3.10]{Ore}, proved that for all $b \geq 5$ there exists a reduced one-dimensional \CM \ local ring of embedding dimension $b$  whose Hilbert function is not monotone increasing.
Finally in \cite{GR}  we can find similar example with embedding dimension four. A long standing conjecture in theory of Hilbert functions is that
the Hilbert function of a one dimensional  complete intersection  is non-decreasing. Rossi conjectures that a similar result holds for
Gorenstein rings.

In this paper we construct a large class of one dimensional \CM \ local rings $(A,\m)$ with the property that if $M$ is an MCM $A$-module then the Hilbert function of $M$ is non-decreasing. Recall a  \CM \ local ring $(B,\n)$ is said to have \textit{minimal multiplicity }if 
\[
e(B) = 1 + \codim(B).
\]
Our result is
\begin{theorem}\label{main}
Let $B,\n)$ be a two dimensional \CM \ local ring with minimal multiplicity. Let $(A,\m)$ be a one-dimensional \CM \ local ring which is a quotient of $B$. If $M$ is a maximal \CM \ $A$-module then the Hilbert function of $M$ (\wrt \ $\m$) is non-decreasing.
\end{theorem}

We now give examples where our result holds.

\begin{example}
Let $(Q,\n)$ be a regular local ring of dimension three. Let $f_1,f_2 \in  \n^2 $ be an $Q$-regular sequence. Assume $f_1 \in \n^2 \setminus\n^3$. Let $A = Q/(f_1,f_2)$. Then if $M$ is a maximal \CM \ $A$-module then the Hilbert function of $M$ (\wrt \ $\m$) is non-decreasing.
The reason for this is that $B = Q/(f_1)$ has minimal multiplicity.
\end{example}

\begin{example}\label{rat}
Let $(B,\n)$ be a two dimensional local ring with pseudo-rational singularity. Then $B$ has minimal multiplicity, see \cite[5.4]{LT}. In particular if $A = B/P$, $P$ a prime ideal of height one or if $A = B/(x)$ where $x$ is $B$-regular then if $M$ is a maximal \CM \ $A$-module then the Hilbert function of $M$ (\wrt \ $\m$) is non-decreasing.
\end{example}

\begin{example}
There is a large class of one dimensional local rings $(R,\m)$ with minimal multiplicity. For examples Arf rings have this property, \cite[2.2]{L}. Let $B = R[X]_{(\m , X)}$. Then $B$ is a two dimensional \CM \ local ring with minimal multiplicity.
\end{example}

Here is an overview of the contents of the paper. 
In Section two  we introduce notation and discuss a few preliminary facts 
that we need. In section three we prove Theorem \ref{main}.

\section{Preliminaries}
In this paper all rings are Noetherian and all modules considered are assumed to be finitely generated (unless otherwise stated). Let $(A,\m)$ be
 a local ring of dimension $d$ with residue field $k = A/\m$. Let $M$ be
 an
$A$-module. If $m$ is a non-zero
 element of $M$ and if $j$ is the largest integer such that $m \in \m^{j}M$,
then we let $m^*$ denote the image of $m$
 in $\m^{j}M/\m^{j+1}M$. 

   The  formal power series
\[
  H_M(z) = \sum_{n \geq 0} H(M,n)z^n
\]
is called the \emph{Hilbert series } of $M$. It is well known that
it is of the form
\begin{equation*}
 H_M(z)  = \frac{h_M(z)}{(1-z)^r}, \ \text{where}\ \ r = \dim M \ \  \
 \text{and} \ 
 h_M(z)  \in
\mathbb{Z}[z].
\end{equation*}
We call  $h_M(z)$ the \emph{h-polynomial} of $M$. 
 If $f$ is a polynomial we use $f^{(i)}$ to denote its  $ i $-th derivative.
 The integers $ e_i(M) = h^{(i)}_{M} (1)/i!$ for $ i \geq 0 $  are called
the
 \emph{Hilbert coefficients} of $M$. The number $ e(M) = e_0(M) $  is
the \emph{multiplicity}
 of $ M $.

\s \textbf{Base change:}
\label{AtoA'}
 Let $\phi \colon (A,\m) \rt (A',\m')$ be a local ring homomorphism. Assume
   $A'$ is a faithfully flat $A$
algebra with $\m A' = \m'$. Set $\m' = \m A'$ and if
 $N$ is an $A$-module set $N' = N\otimes_A A'$.
 In these cases it can be seen that

\begin{enumerate}[\rm (1)]
\item
$\ell_A(N) = \ell_{A'}(N')$.
\item
 $H(M,n) = H(M',n)$ for all $n \geq 0$.
\item
$\dim M = \dim M'$ and  $\depth_A M = \depth_{A'} M'$.
\item
$\depth G(M) = \depth G(M')$.
\end{enumerate}

 \noindent The specific base changes we do are the following:

(i) $A' = A[X]_S$ where $S =  A[X]\setminus \m A[X]$.
The maximal ideal of $A'$ is $\n = \m A'$.
The residue
field of $A'$ is $K = k(X)$. 

(ii) $A' = \widehat{A}$ the completion of $A$ with respect to the maximal ideal.

Thus we can assume that our ring $A$ is complete with infinite residue field.

\textbf{I:} \textit{$L_i(M)$} \\
Let $(A,\m)$ be a Noetherian local ring and $M$ a 
$A$-module. We simplify a construction from \cite{Pu2}.

\s Set $L_0(M) = \bigoplus_{n \geq 0} M/\m^{n+1}M$. Let $\R = A[\m u]$ be the \emph{Rees-algebra} of $\m$. Let $\Sc = A[u]$. Then $\R$ is a subring of $\Sc$. Set
$M[u] = M\otimes_A \Sc$ an $\Sc$-module and so an $\R$-module. Let $\R(M) = \bigoplus_{n \geq 0}\m^nM$ be the Rees-module of $M$ with respect to $\m$. We have the following exact sequence of $\R$-modules
\[
0 \rt \R(M) \rt M[u] \rt L_0(M)(-1) \rt 0.
\]
Thus $L_0(M)(-1)$ (and so $L_0(M)$) is an $\R$-module. We note that $L_0(M)$ is \emph{not} a finitely generated $\R$-module. Also note that $L_0(M) = M \otimes_A L_0(A)$.

\s For $i \geq 1$ set 
$$L_i(M) = \Tor^A_i(M, L_0(A)) = \bigoplus_{n \geq 0 } \Tor^A_i(M, A/\m^{n+1}). $$
We assert that $L_i(M)$ is a finitely generated $\R$-module for $i \geq 1$. It is sufficient to prove it for $i = 1$. We tensor the exact sequence $0 \rt \R \rt \Sc \rt L_0(A)(-1) \rt 0$ with $M$ to obtain a sequence of $\R$-modules
\[
0 \rt L_1(M)(-1) \rt \R \otimes_A M \rt M[u] \rt  L_0(M)(-1) \rt 0.
\] 
Thus $ L_1(M)(-1)$ is a $\R$-submodule of $\R \otimes_A M$. The latter module is a finitely generated $\R$-module. It follows that $L_1(M)$ is a finitely generated $\R$-module.

\s \label{dimL1} Now assume that $A$ is \CM \ of dimension $d \geq 1$.  Set $N = \Syz^A_1(M)$ and $F = A^{\mu(M)}$.  We tensor the exact sequence
\[
0 \rt N \rt F  \rt M \rt 0,
\]
with $L_0(A)$ to obtain an exact sequence of $\R$-modules
\[
0 \rt L_1(M) \rt  L_0(N) \rt  L_0(F) \rt  L_0(M) \rt 0.
\]
It is elementary to see that the function $n \rt \ell(\Tor^A_1(M, A/\m^{n+1}))$ is polynomial of degree $\leq d - 1$. By \cite[Corollary II]{IP}
if $M$ is non-free then it is polynomial of degree $d-1$.
 Thus $\dim L_1(M) = d$ if $M$ is non-free. 

\textbf{II:} \emph{Superficial sequences.}

\s An element
$x \in \m$ is said to be \emph{superficial}
  for $M$ if there exists an integer $c > 0$ such that
$$ (\m^{n}M \colon_M x)\cap\m^cM =  \m^{n-1}M \ \text{ for all }\quad n > c.
$$
 Superficial elements always
exist if  $k$ is infinite  \cite[p.\ 7]{Sa}. A
sequence $x_1,x_2,\ldots,x_r$ in a local ring $(A,\m)$ is said to
be a \emph{superficial sequence} for $M$ if $x_1$ is superficial
for $M$ and $x_i$ is superficial for $M/(x_1,\ldots,x_{i-1})M$ for
$2\leq i \leq r$.

We need the following:
\begin{proposition}\label{GL}
Let $(A,\m)$ be a \CM \ ring of dimension $d$ and let $M$ be a \CM \ $A$-module of dimension $r$. Let $x_1,\ldots,x_c$ be an $M$-superficial sequence with $c \leq r$. Assume $x_1^*,\cdots,x_c^*$ is a $G(M)$-regular sequence. Let $\R = A[\m u]$ be the Rees algebra of $\m$. Set $X_i = x_iu \in \R_1$. Then $X_1,\ldots,X_c$ is a $L_0(M)$-regular sequence.
\end{proposition}
\begin{proof}
We prove the result by induction. First consider the case when $c = 1$. Then the result follows from \cite[2.2(3)]{Pu2}.  We now assume that $c \geq 2$ and the result holds for all \CM \ $A$-modules and sequences of length $c-1$. By $c = 1$ result we get that $X_1$ is $L_0(M)$-regular. Let $N = M/(x_1)$. As $x_1^*$ is $G(M)$-regular we get $G(M)/x_1^* G(M) \cong G(N)$. So $x_2^*,\ldots,x_c^*$ is a $G(N)$-regular sequence. Now also note that $L_0(M)/X_1L_0(M) = L_0(N)$. Thus the result follows.
\end{proof}

\section{Proof of Theorem \ref{main}}
In this section we give a proof of Theorem \ref{main}. We also give an example which shows that it is possible for $\depth G(M)$ to be zero.
\begin{proof}[Proof of Theorem \ref{main}]
We may assume that the residue field of $A$ is infinite.
Let $N = \Syz^B_1(M)$. Then $N$ is a maximal \CM \ $B$-module. As $B$ has minimal multiplicity it follows that $N$ also has minimal multiplicity. So $G(N)$ is \CM \ and $\deg h_N(z) \leq 1$, see \cite[Theorem 16]{Pu1}. Set $r = \mu(M)$, $h_B(z) = 1 + hz$ and as $e(N) = re(A)$ we write $h_N(z) = r + c + (rh - c)z $ (here $c$ can be negative). 

Set $F = A^r$. The exact sequence $0\rt N \rt F \rt M \rt 0$ induces an exact sequence 
\begin{equation}\label{1}
0 \rt L_1(M) \rt L_0(N) \rt L_0(F) \rt L_0(M) \rt 0
\end{equation}
 of $\R$-modules.
Let $x_1,x_2$ be an $N \oplus B$-superficial sequence. Then $x_1^*, x_2^*$ is a $G(N)\oplus G(B)$-regular sequence. Set $X_i = x_iu \in \R_1$. Then by \ref{GL} it follows that $X_1,X_2$ is a $L_0(N) \oplus L_0(F)$-regular sequence. By (\ref{1}) it follows that $X_1, X_2$ is also a $L_1(M)$-regular sequence. As $\dim L_1(M) = 2$ (see \ref{dimL1}) it follows that $L_1(M)$ is a \CM \ $\R$-module. Let the Hilbert series of $L_1(M)$ be $l(z)/(1-z)^2$. Then the coefficients of $l(z)$ is non-negative. 

Let $l(z) = l_0 + l_1z + \cdots + l_m z^n$ and let $h_M(z) = h_0 + h_1z + \cdots + h_pz^p$. By (\ref{1}) we get
\begin{align*}
(1-z)l(z) &= h_N(z) - h_F(z) + (1-z)h_M(z), \\
&= r + c + (rh - c)z - r(1 + h z) + (1-z)h_M(z), \\
&= c(1-z) + (1-z)h_M(z).
\end{align*}
It follows that
\[
l(z) = c + h_M(z).
\]
It follows that $m = p$ and $h_i = l_i$ for $i \geq 1$. In particular $h_i \geq 0$ for $i \geq 1$. Also $h_0 = \mu(M) > 0$. Thus $h_M(z)$ has non-negative coefficients. It follows that the Hilbert function of $M$ is non-decreasing. 
\end{proof}

We now give an example which shows that it is possible for $\depth G(M)$ to be zero.
\begin{example} \label{ex2}
Let $K$ be a field and let $A = K[[t^6, t^7, t^{15}]]$. It can be verified that
\[
A \cong \frac{K[[X,Y,Z]]}{(Y^3-XZ, X^5 - Z^2)}
\]
and that
\[
G(A) \cong \frac{K[X,Y,Z]}{(XZ, Y^6, Y^3Z,Z^2)}
\]
Note that $ZY^2$ annihilates $(X,Y,Z)$. So $\depth G(A) = 0$.

Set $B  = K[[X,Y,Z]]/(Y^3-XZ)$. Then $B$ is a two-dimensional \CM \ ring with minimal multiplicity and $A$ is a one-dimensional \CM \ quotient of $B$. Set $M = A$.
\end{example}

\end{document}